\documentclass[12pt,a4paper]{amsart}

\newtheorem{theorem}{Theorem}[section]
\newtheorem{corollary}[theorem]{Corollary}
\newtheorem{lemma}[theorem]{Lemma}
\newtheorem{proposition}[theorem]{Proposition}

\usepackage{amsfonts}
\usepackage[figuresright]{rotating}
\usepackage{amssymb}
\usepackage{amsmath}
\usepackage{amsmath}
\usepackage{graphics}
\title{Some relatives of the Catalan sequence}
\author{El\.zbieta Liszewska}
\author{Wojciech M{\l}otkowski}
\thanks{W.~M. is supported by the Polish
National Science Center grant No. 2016/21/B/ST1/00628.}
\address{Telecommunications and Teleinformatics Department (KTT),	
Wroc{\l}aw University of Science and Technology,
Wybrze\.{z}e Wyspia\'{n}skiego~27,
50-370 Wroc{\l}aw, Poland}
\address{Instytut Matematyczny,
Uniwersytet Wroc{\l}awski,
Plac~Grunwaldzki~2/4,
50-384 Wroc{\l}aw, Poland}
\email{elzbieta.liszewska@pwr.edu.pl}
\email{mlotkow@math.uni.wroc.pl}
\begin{document}

\begin{abstract}
We study a family of sequences $c_n(a_2,\ldots,a_r)$, where $r\ge2$ and $a_2,\ldots,a_r$ are real parameters.
We find a sufficient condition for positive definiteness of the sequence $c_n(a_2,\ldots,a_r)$
and check several examples from OEIS. We also study relations
of these sequences with the free and monotonic convolution.
\end{abstract}

\maketitle

\section*{Introduction}

The Catalan sequence $\binom{2n+1}{n}\frac{1}{2n+1}$:
\[
1, 1, 2, 5, 14, 42, 132, 429, 1430, 4862, 16796,\ldots
\]
(entry $A000108$ in OEIS \cite{oeis}) plays an important role in mathematics.
It counts several combinatorial objects, such as binary trees, noncrossing partitions, Dyck paths,
dissections of a polygon into triangles and many others, as can be seen in the Stanley's book~\cite{stanley2015}.
It also shows up as the moment sequence of the Marchenko-Pastur law
and (in its aerated version) of the Wigner law, which in the free probability theory
are the analogs of the Poisson and the normal law.

One of the generalizations are Fuss numbers $\binom{pn+1}{n}\frac{1}{pn+1}$,
which count for example $p$-ary trees with $np-n+1$ leaves,
see~\cite{gkp}.
The generating function
\begin{align}
\mathcal{B}_{p}(z):&=\sum_{n=0}^{\infty}\binom{np+1}{n}\frac{z^n}{np+1}\\
\intertext{satisfies equality}
\mathcal{B}_{p}(z)&=1+z\mathcal{B}_{p}(z)^{p}.\label{fussgeneratingformula}\\
\intertext{Lambert observed that the powers of $\mathcal{B}_{p}(z)$ admit quite nice
Taylor expansion:}
\mathcal{B}_{p}(z)^{r}&=\sum_{n=0}^{\infty}\binom{np+r}{n}\frac{r z^n}{np+r}.\label{lambertformula}
\end{align}
The coefficients $\binom{np+r}{n}\frac{r}{np+r}$ have also combinatorial interpretations, see~\cite{gkp},
and are called generalized Fuss numbers or Raney numbers. Formulas (\ref{fussgeneratingformula}), (\ref{lambertformula})
remain true if the parameters $p,r$ are real.

It turns out that for $p,r\in\mathbb{R}$ the sequence $\binom{np+r}{n}\frac{r}{np+r}$
is positive definite if and only if either $p\ge1, 0<r\le p$ or $p\le0,p-1\le r<0$ or $r=0$,
see \cite{mlotkowski2010,mlopezy2013,mlotkowskipenson2014binomial,forresterliu2014,liupego2014}.
Moreover, the corresponding probability distributions $\mu(p,r)$ are interesting from the point of view
of noncommutative probability, for example the free $R$-transform of $\mu(p,r)$ is $\mathcal{B}_{p-r}(z)^{r}-1$,
$\mu(1+p,1)=\mu(1,1)^{\boxtimes p}$, $p>0$, where ``$\boxtimes$" is the multiplicative free convolution,
and $\mu(p,q)\rhd\mu(p+r,r)=\mu(p+r,q+r)$, where ``$\rhd$" is the monotonic convolution, see \cite{muraki2001}.

In this paper we are going to study sequences $c_n(\mathbf{a})$ which are defined
by real parameters $\mathbf{a}=(a_2,\ldots,a_r)$ and recurrence relation (\ref{recurrenceforc}),
with $c_0(\mathbf{a})=1$. If $r=2$, $a_2=1$ then $c_n(\mathbf{a})$ is the Catalan sequence.
First we provide combinatorial motivations for such sequences, for example in terms of action of a finite family
of operators on a product.
Formula (\ref{recurrenceforc}) implies equations~(\ref{equationforcfunction})
and (\ref{equationfordfunction}) for the generating  and
the upshifted generating function $C_{\mathbf{a}}(z)$, $D_{\mathbf{a}}(z):=zC_{\mathbf{a}}(z)$.

In Section~\ref{sectiongenerating} we study these functions, in particular we give a sufficient condition when the domain of $D_{\mathbf{a}}(z)$
and $C_{\mathbf{a}}(z)$ is contained in $\mathbb{C}\setminus\mathbb{R}$.
Consequently, $D_{\mathbf{a}}$ is a Pick function and the sequence $c_n(\mathbf{a})$ is positive definite.
If this is the case then the corresponding probability distribution will be denoted $\mu(\mathbf{a})$.
We find the free $R$-transform for $\mu(\mathbf{a})$,
which is a rational function, and study the convolution semigroup
$\mu(\mathbf{a})^{\boxplus t}$. We also prove that $\mu(\mathbf{a}')\rhd\mu(\mathbf{a}'')=\mu(\mathbf{a})$,
where the sequence $\mathbf{a}$ is given by (\ref{monotonicconvolutionp}).

In Section~\ref{sectionr3} we study in details the  case $r=3$. We give formulas for $c_n(a,b)$,
free cumulants $\kappa_n(a,b)$, the generating functions $C_{a,b}(z),D_{a,b}(z)$,
and prove that the sequence $c_n(a,b)$ is positive definite if and only if $a^2+3b\ge0$.
Moreover, for $a,b>0$ the distribution $\mu(a,b)$ is infinitely divisible with respect to the additive free
convolution ``$\boxplus$".

For the case $r=4$ we provide sufficient conditions for positive definiteness
(Theorem~\ref{theoremr4realroots}, Proposition~\ref{propositionr4}).
Then we consider the symmetric case, i.e. when $a_j=0$ whenever $j$ is even, and
in the final section we provide a record
of these integer sequences from OEIS which
are of the form $c_n(\mathbf{a})$ and verify their positive definiteness.

\section{Operators on a set}\label{sectionoperators}

Suppose that $X$ is a set and that for $j=2,3,\ldots,r$ we are given some $j$-ary operators on $X$,
$F^{(j)}_{i}:X^j\to X$, $i=1,2,\ldots,a_j$, here $a_j$ are positive integers.
Put  $\mathbf{a}:=(a_2,\ldots,a_r)$ and denote by $c_{n}(\mathbf{a})$ the number of all possible
compositions of these operations applied to the product $x_0x_1\ldots x_n$.
For example, for $n=2$ we can apply either
\[
F^{(3)}_{i}(x_0, x_1, x_2)\quad\hbox{or}\quad
F^{(2)}_{i_1}(F^{(2)}_{i_2}(x_0, x_1),x_2)\quad\hbox{or}\quad
F^{(2)}_{i_1}(x_0,F^{(2)}_{i_2}(x_1, x_2)),
\]
where
$1\le i\le a_3$, $1\le i_1,i_2\le a_2$, so that $c_2(\mathbf{a})=a_3+2a_2^2$.
Each such composed operator on $x_0 x_1\ldots x_n$ can be written as
\[F^{(j_0)}_{i_0}(w_1, w_2,\ldots,w_{j_0}),\quad j_0=2,\ldots, r,\,\, i_0=1,\ldots,a_{j_0},\]
where $w_s$ is a result of a composition of operators
$F^{(j)}_{i}$ on $x_{p_{s-1}} x_{p_{s-1}+1} \ldots x_{p_{s}-1}$,
with $0=p_0<p_1<\ldots<p_{j_0}=r+1$.
Accordingly, we obtain the following recurrence relation
\begin{equation}\label{recurrenceforc}
c_{n}(\mathbf{a})=\sum_{j=2}^{r} a_j \sum_{\substack{u_1,\ldots,u_j\ge0\\u_1+\ldots+u_j=n-j+1}}
c_{u_1}(\mathbf{a})c_{u_2}(\mathbf{a})\ldots c_{u_j}(\mathbf{a}),
\end{equation}
for $n\ge1$, with the initial condition $c_0(\mathbf{a})=1$.

In a similar way one can prove that $c_n(\mathbf{a})$ is the number of:
\begin{itemize}
\item finite rooted labeled trees with $n+1$ leaves such that the nodes are of order less or equal to $r$,
and each node $v$ of order $2\le j\le r$ is given a label $\ell(v)\in\{1,2,\ldots,a_j\}$,

\item labeled sequences $\left(u_1^{\epsilon_1},\ldots,u_{s}^{\epsilon_s}\right)$,
such that $s\ge1$,
\[
u_k\in\{1,-1,-2,\ldots,1-r\},\quad u_1+\ldots+u_k>0,
\]
for $1\le k\le s$, $u_1+\ldots+u_s=1$ and if $u_k=1-j$ then $\epsilon_k\in\{1,\ldots,a_j\}$, where $a_{0}:=1$,
and $|\{j:u_j=1\}|=n+1$,

\item Dyck paths from $(0,0)$ to $(s,1)$, with labeled steps $(1,u_j^{\epsilon_{j}})$,
where $s$, $u_j$, $\epsilon_j$ satisfy the same conditions as in the previous point.
\end{itemize}

From (\ref{recurrenceforc}) the generating function
\begin{equation}\label{generatingfunctionc}
C_{\mathbf{a}}(z):=\sum_{n=0}^{\infty} c_{n}(\mathbf{a})z^n
\end{equation}
satisfies equation
\begin{equation}\label{equationforcfunction}
C_{\mathbf{a}}(z)=1+a_2 z C_{\mathbf{a}}(z)^2+%a_3 z^2 C_{\mathbf{a}}(z)^3+
\ldots+a_r z^{r-1} C_{\mathbf{a}}(z)^{r}.
\end{equation}
We will also work with the upshifted generating function $D_{\mathbf{a}}(z):=z C_{\mathbf{a}}(z)$,
so that (\ref{equationforcfunction}) is equivalent to
\begin{equation}\label{equationfordfunction}
D_{\mathbf{a}}(z)=z+a_2 D_{\mathbf{a}}(z)^2+\ldots +a_r D_{\mathbf{a}}(z)^r,
\end{equation}
with $D_{\mathbf{a}}(0)=0$. Hence $D_{\mathbf{a}}(z)$ is the inverse function of
$P_{\mathbf{a}}(w):=w-a_2 w^2-\ldots-a_r w^r$ in a neighborhood of $w=0$.

\section{The generating functions and positive definiteness}\label{sectiongenerating}

Motivated by the previous section we are going to study sequences defined by the recurrence
relation (\ref{recurrenceforc}), with $c_0(\mathbf{a}):=1$, where
$\mathbf{a}=(a_2,\ldots,a_r)\in\mathbb{R}^{r-1}$, $r\ge2$, with $a_r\ne0$.
Then the generating function $C_{\mathbf{a}}(z)$, given by (\ref{generatingfunctionc}),
and the upshifted generating function $D_{\mathbf{a}}(z):=z C_{\mathbf{a}}(z)$,
satisfy relations (\ref{equationforcfunction}) and (\ref{equationfordfunction}),
with $D_{\mathbf{a}}(0)=0$.

First we note the following homogeneity:

\begin{proposition}
If $d\ne0$ and $\mathbf{b}=(da_2,d^2a_3,\ldots,d^{r-1} a_{r})$ then
\[
d^n c_n(\mathbf{a})=c_n(\mathbf{b}),\quad
C_{\mathbf{a}}(dz)=C_{\mathbf{b}}(z),\quad
D_{\mathbf{a}}(dz)=d D_{\mathbf{b}}(z).
\]
\end{proposition}

\begin{proof}
This is a direct consequence of (\ref{recurrenceforc}) and (\ref{generatingfunctionc}).
\end{proof}

\begin{proposition}
The functions $C_{\mathbf{a}}(z)$, $D_{\mathbf{a}}(z)$ can not be extended to entire functions on $\mathbb{C}$.
\end{proposition}

\begin{proof}
Assume that $D_{\mathbf{a}}(z)$ is an entire function.
It can not be a polynomial for otherwise each side of (\ref{equationfordfunction})
would be a polynomial of different degree. Hence $D(z)$ is essentially singular at infinity.
By the Casorati-Weierstrass theorem, there exists a sequence $z_n\in\mathbb{C}$
such that $|z_n|\to\infty$ and $D_{\mathbf{a}}(z_n)\to0$, however this is impossible due to (\ref{equationfordfunction}).
Therefore $D_{\mathbf{a}}(z)$, and consequently $C_{\mathbf{a}}(z)$, can not be extended to entire functions.
\end{proof}

Now we will study a possible domain of the functions $C_{\mathbf{a}}(z)$, $D_{\mathbf{a}}(z)$.
Denote by $\mathcal{N}_{\mathbf{a}}$ the set of all such $z\in\mathbb{C}$ that the equation
\begin{equation}\label{equationwz}
w=z+a_2 w^2+\ldots+a_r w^r
\end{equation}
has a multiple solution. Put
\begin{equation}\label{polynomialpa}
P_{\mathbf{a}}(w):=w-a_2 w^2-\ldots-a_r w^r,
\end{equation}
so that $z\in\mathcal{N}_{\mathbf{a}}$ if and only if $P_{\mathbf{a}}(w)-z$ has a multiple root.

\begin{proposition}
A complex number $z_0$ belongs to $\mathcal{N}_{\mathbf{a}}$ if and only if
there exists $w_0\in\mathbb{C}$ such that $P_{\mathbf{a}}(w_0)=z_0$ and $P_{\mathbf{a}}'(w_0)=0$.
\end{proposition}

\begin{proof}
If $z_0\in\mathcal{N}_{\mathbf{a}}$ then $P_{\mathbf{a}}(w)-z_0=(w-w_0)^2 Q(w)$ for some
$w_0\in\mathbb{C}$ and a polynomial $Q(w)$. Then $P_{\mathbf{a}}(w_0)=z_0$ and $P_{\mathbf{a}}'(w_0)=0$.

On the other hand, if $P_{\mathbf{a}}(w_0)=z_0$ and $P_{\mathbf{a}}'(w_0)=0$
then $P_{\mathbf{a}}(w)-z_0=(w-w_0)Q_1(w)$ for a certain polynomial $Q_1$,
and the assumption $P_{\mathbf{a}}'(w_0)=0$ implies that $Q_1(w_0)=0$, hence $z_0\in\mathcal{N}_{\mathbf{a}}$.
\end{proof}

\begin{corollary}\label{corollaryrealroots}
If all the roots of $P_{\mathbf{a}}'(w)$ are real then $\mathcal{N}_{\mathbf{a}}\subseteq\mathbb{R}$.
\end{corollary}

Note that cardinality of $\mathcal{N}_{\mathbf{a}}$ is at most $r-1$
and if $z\in\mathcal{N}_{\mathbf{a}}$ then $\overline{z}\in\mathcal{N}_{a}$.
In fact, $\mathcal{N}_{\mathbf{a}}$ consists of all those $z_0\in\mathbb{C}$ for which
some of the branches $w_0(z),\ldots,w_{r-1}(z)$ of the set of solutions of (\ref{equationwz})
meet, while we are interested in just one of these branches, namely that
which satisfies $w_0(0)=0$.
It may happen that $0\in\mathcal{N}_{\mathbf{a}}$, but $w=0$
is a single root of the equation $P_{\mathbf{a}}(w)=0$.
Therefore we can uniquely define $D_{\mathbf{a}}(z)$ as an analytic function at least on the set
\[
\mathcal{D}_{\mathbf{a}}:=\mathbb{C}\setminus\{ t z_0:z_0\in\mathcal{N}_{\mathbf{a}}, z_0\ne0, t\ge1\}
\]
in such a way that (\ref{equationfordfunction}) is satisfied, $D_{\mathbf{a}}(0)=0$ and $D_{\mathbf{a}}\left(\overline{z}\right)=\overline{D_{\mathbf{a}}(z)}$
for $z\in\mathcal{D}_{\mathbf{a}}$.
Note that (\ref{equationfordfunction}) implies that $D_{\mathbf{a}}(z)$ is injective on its domain.

Recall that a sequence $\left\{c_n\right\}_{n=0}^{\infty}$ of real numbers
is called \textit{positive definite} if
\[
\sum_{i,j\ge0}c_{i+j}x_i x_j\ge0
\]
holds for every sequence of real numbers $x_i$ with finite number of nonzero terms.
Equivalently, there exists a positive measure $\mu$ on $\mathbb{R}$
such that $c_n$ are moments of $\mu$, i.e. $c_n=\int_{\mathbb{R}} t^n\mu(dt)$
for $n=0,1,\ldots$.
The moment generating function, the upshifted moment generating function and the Cauchy
transform of $\mu$ are given by
\begin{equation}
C_{\mu}(z):=\int_{\mathbb{R}}\frac{\mu(dt)}{1-t z},\quad
D_{\mu}(z):=\int_{\mathbb{R}}\frac{z \mu(dt)}{1-t z},\quad
G_{\mu}(z):=\int_{\mathbb{R}}\frac{\mu(dt)}{z-t},
\end{equation}
so that $D_{\mu}(z)=z C_{\mu}(z)$ and $D_{\mu}(z)=G_{\mu}(1/z)$.

We are now interested for which real parameters $a_2,\ldots,a_r$ the sequence $c_{n}(\mathbf{a})$,
defined by (\ref{recurrenceforc}), with $c_0(\mathbf{a})=1$,
is positive definite.

\begin{theorem}\label{theorempositivedefinite}
If $\mathcal{N}_{\mathbf{a}}\subseteq\mathbb{R}$ then the sequence $c_n(\mathbf{a})$ is positive definite.
\end{theorem}

\begin{proof}
By the assumptions, $D_{\mathbf{a}}(z)$ can be defined as an analytic function
on $\mathcal{D}_{\mathbf{a}}\subseteq\mathbb{C}\setminus\mathbb{R}$
in such a way that $D_{\mathbf{a}}(\overline{z})=\overline{D_{\mathbf{a}}(z)}$ and $D_{\mathbf{a}}(0)=0$.
For $w=x+y\mathrm{i}$
\begin{multline*}
\Im\left(w-a_2 w^2-\ldots-a_r w^r\right)\\
=y\left(1-\sum_{k=2}^{r} a_k \sum_{j=0}^{\lfloor(k-1)/2\rfloor}%{\lfloor\frac{k-1}{2}\rfloor}
\binom{k}{2j+1}x^{k-2j-1}y^{2j}(-1)^j\right)
\end{multline*}
so that for $|z|$ small, $z\in\mathbb{C}\setminus\mathbb{R}$, the sign of $\Im D_{\mathbf{a}}(z)$ coincides with that of $\Im z$.
Now one can observe from (\ref{equationfordfunction}) that $\Im D_{\mathbf{a}}(z)$ never vanish on $\mathbb{C}\setminus\mathbb{R}$.
Therefore if $\Im z>0$ (resp. $<0$) then $\Im D_{\mathbf{a}}(z)>0$ (resp. $<0$), i.e. $D_{\mathbf{a}}(z)$ is a Pick function,
which implies that $D_{\mathbf{a}}(1/z)$ is the Cauchy transform of a probability measure on $\mathbb{R}$.
\end{proof}

\begin{corollary}\label{corollarypositivedefinite}
If all the roots of $P_{\mathbf{a}}'(w)$ are real then $c_n(\mathbf{a})$ is positive definite.
\end{corollary}

If the sequence $c_n(\mathbf{a})$ is positive definite then we will denote
by $\mu(\mathbf{a})$ the probability measure for which the numbers $c_n(\mathbf{a})$ are moments.
Since the domain of the generating function $C_{\mathbf{a}}(z)$ contains a neighborhood of $0$,
such a measure is unique and has compact support.

\textbf{Example} (A063020). Take $\mathbf{a}=(1,1,-1)$.
Then $P_{\mathbf{a}}(w)=w-w^2-w^3+w^4$ and its derivative $P_{\mathbf{a}}'(w)=(w-1)(4w^2+w-1)$
has only real roots. Therefore the sequence A063020:
\[
1, 1, 3, 9, 32, 119, 466, 1881, 7788, 32868, 140907, 611871,\ldots
\]
(with the initial $0$ term skipped) is positive definite.

The following example shows that the converse of Corollary~\ref{corollaryrealroots} is not true.

\textbf{Example} (A121988).
Take $\mathbf{a}=(2,-2,1)$. Then $P_{\mathbf{a}}(w)=w-2w^2+2w^3-w^4$,
$P_{\mathbf{a}}'(w)=1-4w+6w^2-4w^3$. The roots of $P_{\mathbf{a}}'(w)$ are $1/2,1/2\pm \mathrm{i}/2$
and $P_{\mathbf{a}}(1/2\pm \mathrm{i}/2)=1/4$, so $\mathcal{N}_{2,-2,1}=\{3/16,1/4\}$ and the sequence A121988:
\[
1, 2, 6, 21, 80, 322, 1348, 5814, 25674, 115566,\ldots
\]
is positive definite.

\section{$R$-transform and free additive convolution semigroups}\label{sectionfree}
%\textbf{Conjecture:} For every $n\ge1$ the polynomial $c_{n}(1,t)$ has only real and single roots.

For a sequence $1=c_0,c_1,c_2\ldots$, with generating function
$C(z):=\sum_{n=0}^{\infty} c_n z^n$ (possibly a formal power series),
the \textit{free $R$-transform} of $c_n$, or of $C(z)$, is defined by
\begin{equation}\label{freertransformdefinition}
1+R(zC(z))=C(z).
\end{equation}
The coefficients $\kappa_n$ in the Taylor expansion $R(z)=\sum_{n=1}^{\infty} \kappa_n z^n$
are called \textit{free cumulants} of the sequence $c_n$.
If $c_n$ are moments of a probability distribution on $\mathbb{R}$, i.e. $c_n=\int_{\mathbb{R}} x^n\mu(dx)$,
$n=0,1,\ldots$, then $R$ (resp $\kappa_n$) is called the \textit{free $R$-transform} (resp. the \textit{free cumulants}) of $\mu$.
If $R_1(z),R_2(z),R(z)$ are $R$-transforms of probability distributions $\mu_1,\mu_2,\mu$, respectively,
and if $t\ge1$ then $R_1(z)+R_2(z)$, $t R(z)$ are $R$-transforms of certain probability
distributions which are denoted $\mu_1\boxplus\mu_2$, $\mu^{\boxplus t}$, respectively,
see \cite{vdn,nicaspeicherlectures,mingospeicher} for details.

Now take $\mathbf{a}=(a_2,\ldots,a_r)\in\mathbb{R}^{r-1}$, with $r\ge2$, $a_r\ne0$.

\begin{proposition}\label{propositionrtransform}
The $R$ transform of the sequence $c_n(\mathbf{a})$ is
\begin{equation}\label{freertransform}
R_{\mathbf{a}}(z)=\frac{a_2 z+ \ldots+a_r z^{r-1}}{1-a_2 z-\ldots-a_r z^{r-1}}.
\end{equation}
\end{proposition}

\begin{proof}
It is sufficient to note that equation (\ref{equationforcfunction}) can be written as
\[
1+ \frac{a_2 zC_{\mathbf{a}}(z)+\ldots+a_{r} z^{r-1} C_{\mathbf{a}}(z)^{r-1}}{1-a_2 zC_{\mathbf{a}}(z)-\ldots - a_{r} z^{r-1} C_{\mathbf{a}}(z)^{r-1}}=C_{\mathbf{a}}(z).\]
\end{proof}

Denote by $c_n^{\boxplus t}(\mathbf{a})$ the sequence for which the $R$-transform is $t\cdot R_{\mathbf{a}}(z)$.
The generating and the upshifted generating functions will be denoted $C_{\mathbf{a}}^{\boxplus t}(z),D_{\mathbf{a}}^{\boxplus t}(z)$.
If $c_n^{\boxplus t}(\mathbf{a})$ is positive definite then the corresponding probability distribution
on $\mathbb{R}$ will be denoted $\mu(\mathbf{a})^{\boxplus t}$.
If $0<t_1\le t_2$ and $c_n^{\boxplus t_1}(\mathbf{a})$ is positive definite then so is $c_n^{\boxplus t_2}(\mathbf{a})$
and $\mu(\mathbf{a})^{\boxplus t_2}=\left(\mu(\mathbf{a})^{\boxplus t_1}\right)^{\boxplus t_2/t_1}$.

\begin{proposition}\label{propositionfreepower}
The function $D_{\mathbf{a}}^{\boxplus t}(z)$ satisfies
\begin{equation}\label{freepowerformulad}
z=\frac{D_{\mathbf{a}}^{\boxplus t}(z)-a_2 D_{\mathbf{a}}^{\boxplus t}(z)^2-\ldots-a_r D_{\mathbf{a}}^{\boxplus t}(z)^r}
{1+(t-1)\left(a_2 D_{\mathbf{a}}^{\boxplus t}(z)+\ldots+a_r D_{\mathbf{a}}^{\boxplus t}(z)^{r-1}\right)}
\end{equation}
so that $D_{\mathbf{a}}^{\boxplus t}(z)$ is the composition inverse function to
\begin{equation}\label{freepowerformulaw}
z=\frac{w-a_2 w^2-\ldots-a_r w^r}
{1+(t-1)\left(a_2 w+\ldots+a_r w^{r-1}\right)}.
\end{equation}
\end{proposition}

\begin{proof}
By (\ref{freertransformdefinition}) and (\ref{freertransform}) we have
\[
1+t\frac{a_2 D_{\mathbf{a}}^{\boxplus t}(z)+\ldots
+a_r D_{\mathbf{a}}^{\boxplus t}(z)^{r-1}}{1-a_2 D_{\mathbf{a}}^{\boxplus t}(z)-\ldots-a_r D_{\mathbf{a}}^{\boxplus t}(z)^{r-1}}
=\frac{D_{\mathbf{a}}^{\boxplus t}(z)}{z}
\]
which leads directly to the formula.
\end{proof}

\section{Compositions and monotonic convolution}\label{sectionmonotonic}

For compactly supported probability distributions $\mu_1,\mu_2$ on $\mathbb{R}$,
with the moment generating functions
\[
C_{\mu_i}(z):=\int_{\mathbb{R}}\frac{\mu_i(dx)}{1-xz}=\sum_{n=0}^{\infty}\int_{\mathbb{R}} x^n\mu_i(dx)
\]
there exists unique compactly supported probability distribution $\mu_1\rhd\mu_2$ on $\mathbb{R}$,
called \textit{monotonic convolution} of $\mu_1$ and $\mu_2$, such that its moment generating function satisfies
\begin{equation}\label{monotonicc}
C_{\mu_1\rhd\mu_2}(z)=C_{\mu_1}\big(z C_{\mu_2}(z)\big)\cdot C_{\mu_2}(z).
\end{equation}
For details we refer to \cite{muraki2001}. In terms of the upshifted moment generating functions
$D_{\mu_i}(z):=z C_{\mu_i}(z),D_{\mu_1\rhd\mu_2}(z):=zC_{\mu_1\rhd\mu_2}(z)$ relation (\ref{monotonicc}) becomes
\begin{equation}\label{monotonicd}
D_{\mu_1\rhd\mu_2}(z)=D_{\mu_1}\big(D_{\mu_2}(z)\big).
\end{equation}

\begin{proposition}
Suppose that $\mathbf{a}'=(a_2',\ldots,a_{r'}')$, $\mathbf{a}''=(a_2'',\ldots,a_{r''}'')$
and the sequences $c_n(\mathbf{a}'),c_n(\mathbf{a}'')$ are positive definite,
with the corresponding probability distributions $\mu(\mathbf{a}'),\mu(\mathbf{a}'')$. Then
\[
\mu(\mathbf{a}')\rhd\mu(\mathbf{a}'')=\mu(\mathbf{a}),
\]
where the sequence $\mathbf{a}$ is given by
\begin{equation}\label{monotonicconvolutionp}
P_{\mathbf{a}}(w)=P_{\mathbf{a}''}\big(P_{\mathbf{a}'}(w)\big).
\end{equation}
\end{proposition}

\begin{proof}
It is a consequence of (\ref{monotonicd}) and the fact that $D_{\mathbf{a}}(z)$
is the composition inverse of $P_{\mathbf{a}}(w)$.
\end{proof}

It is natural to define $\mathbf{a}'\rhd\mathbf{a}'':=\mathbf{a}$ by formula (\ref{monotonicconvolutionp}).

\begin{corollary}\label{corollarymonotonic}
If the sequences $c_n(\mathbf{a}'),c_n(\mathbf{a}'')$ are positive definite
then so is $c_n(\mathbf{a}'\rhd\mathbf{a}'')$.
\end{corollary}

\textbf{Example:} Take $\mathbf{a}'=(1)$, $\mathbf{a}''=(1,2)$.
Then $P_1(w)=w-w^2$, $P_{1,2}(w)=w-w^2-2w^3$ and
\[
P_{1,2}(P_1(w))=w-2w^2+5w^4-6w^5+2w^6=P_{2,0,-5,6,-2}(w),
\]
so that
\[
(1)\rhd(1,2)=(2,0,-5,6,-2).
\]
This leads to the positive definite sequence $c_n(2,0,-5,6,-2)$:
\[
1, 2, 8, 35, 170, 866, 4580, 24852, 137560, 773278,\ldots,
\]
which is absent in OEIS.
The roots of
\[
P_{2,0,-5,6,-2}'(w)=1-4w+20w^3-30w^4+12w^5\]
are
\[
\frac{1}{2},\qquad
\frac{1}{2}\pm\frac{1}{6}\sqrt{6\sqrt{7}+15},\qquad
\frac{1}{2}\pm\frac{\mathrm{i}}{6}\sqrt{6\sqrt{7}-15}
\]
and
\[
\mathcal{N}_{2,0,-5,6,-2}=\left\{\frac{-7\sqrt{7}-10}{54},\frac{7\sqrt{7}-10}{54},\frac{5}{32}\right\}.
\]
This example illustrates that the operation ``$\rhd$" does not preserve real rootedness
of $P_{\mathbf{a}}'$, but, as we will prove, it does preserve the property ``$\mathcal{N}_{\mathbf{a}}\subseteq\mathbb{R}$".

\begin{proposition}
For $\mathbf{a}'=(a_2',\ldots,a_{r'}')$, $\mathbf{a}''=(a_2'',\ldots,a_{r''}'')$ we have
\begin{equation}\label{monotonicna}
\mathcal{N}_{\mathbf{a'\rhd\mathbf{a''}}}=
\mathcal{N}_{\mathbf{a''}}\cup P_{\mathbf{a}''}\left[\mathcal{N}_{\mathbf{a'}}\right].
\end{equation}
Consequently, if $\mathcal{N}_{\mathbf{a'}}\subseteq\mathbb{R}$, $\mathcal{N}_{\mathbf{a''}}\subseteq\mathbb{R}$
then $\mathcal{N}_{\mathbf{a'}\rhd\mathbf{a''}}\subseteq\mathbb{R}$.
\end{proposition}

\begin{proof}
It suffices to observe that
\[
\left\{P_{\mathbf{a}''}(P_{\mathbf{a}'}(w)):w\in\mathbb{C},P_{\mathbf{a}''}'(P_{\mathbf{a}'}(w))=0\right\}
=\mathcal{N}_{\mathbf{a''}}
\]
(as $P_{\mathbf{a}'}$ is a function \textit{onto} $\mathbb{C}$) and
\[
\left\{P_{\mathbf{a}''}(P_{\mathbf{a}'}(w)):w\in\mathbb{C},P_{\mathbf{a}'}'(w)=0\right\}
=P_{\mathbf{a}''}\left[\mathcal{N}_{\mathbf{a'}}\right],
\]
which proves (\ref{monotonicna}).
\end{proof}

\section{The case $r=2$: Catalan numbers}\label{sectionr2}

For $r=2$, $a\in\mathbb{R}\setminus\{0\}$  we have
\begin{align*}
c_{n}(a)&= a \sum_{\substack{u_1,u_2\ge0\\u_1+u_2=n-1}}
c_{u_1}(a)c_{u_2}(a),\\
C_{a}(z)&=1+az C_{a}(z)^2,
\end{align*}
which leads to the dilated Catalan numbers
\[
c_n(a)=\binom{2n+1}{n}\frac{a^n}{2n+1},
\]
with the generating function
\[
C_{a}(z)=\frac{1-\sqrt{1-4az}}{2az}=\frac{2}{1+\sqrt{1-4az}}.
\]

It is well know that the Catalan sequence is positive definite, namely
\[
\binom{2n+1}{n}\frac{1}{2n+1}=\frac{1}{2\pi}\int_{0}^{4}x^n\sqrt{\frac{4-x}{x}}\,dx,
\]
$n=0,1,\ldots$.
Taking $P(w):=w-w^2$ we can now prove this applying Theorem~{\ref{theorempositivedefinite}}.

It is also known that $\mu(1)$ (and hence $\mu(a)$, with $a\ne0$) is infinitely divisible with respect to the additive free
convolution $\boxplus$
and the moments of $\mu(1)^{\boxplus t}$ are the Narayana polynomials (see $A001263$ in OEIS):
\[
c_n^{\boxplus t}(1)=\sum_{k=1}^{n}\binom{n}{k-1}\binom{n}{k}\frac{t^k}{m}
\]
for $n\ge1$, see \cite{nicaspeicherlectures}.
The generating function satisfies
\[
C_{1}^{\boxplus t}(z)=1+(t-1)z C_{1}^{\boxplus t}(z)+z C_{1}^{\boxplus t}(z)^2,
\]
and $\mu(1)^{\boxplus t}$ are known as the Marchenko-Pastur distributions:
\[
\mu(1)^{\boxplus t}=\max\{1-t,0\}\delta_{0}+\frac{\sqrt{4t-(x-1-t)^2}}{2\pi x}\chi_{(1+t+2\sqrt{t},1+t-2\sqrt{t})}(x)\,dx.
\]

\section{The case $r=3$}\label{sectionr3}

In this section we will confine ourselves to the case $r=3$.
Put $a:=a_2$, $b:=a_3\ne0$. The corresponding sequence, polynomial and the generating
functions will be denoted $c_n(a,b)$, $P_{a,b}(w)$, $C_{a,b}(z)$, $D_{a,b}(z)$.

\subsection{Formula for $c_n(a,b)$}

\begin{proposition}
For $a,b\in\mathbb{R}$, $n\ge0$, we have
\begin{equation}\label{r3formulac}
c_{n}(a,b)=\frac{1}{n+1}\sum_{j=0}^{\lfloor n/2\rfloor}\binom{2n-j}{n}\binom{n-j}{j}a^{n-2j}b^{j}.
\end{equation}
\end{proposition}

\begin{proof}
By the Lagrange inversion theorem:
\[
D_{a,b}(z)=
\sum_{n=1}^{\infty}\dfrac{d^{n-1}}{dw^{n-1}}\left.\left(\dfrac{w}{w-aw^2-bw^3}\right)^n\right|_{w=0}\dfrac{z^n}{n!}.
\]
Since $\binom{-n}{k}(-1)^k=\binom{n+k-1}{k}$, we have
\begin{align*}
\left(\dfrac{w}{w-aw^2-bw^3}\right)^n&=(1-aw-bw^2)^{-n}\\
&=\sum_{k=0}^{\infty}{n+k-1\choose k}(aw+bw^2)^{k}\\
&=\sum_{k=0}^{\infty}{n+k-1\choose k}
\sum_{j=0}^{k}{k\choose j}a^{k-j}b^jw^{k+j}\\
&=\sum_{m=0}^{\infty}w^{m}
\sum_{j=0}^{\lfloor{m/2}\rfloor}{n+m-j-1\choose m-j}{m-j\choose j}a^{m-2j}b^{j}.
\end{align*}

The coefficient of $D_{a,b}(z)$ at $w^{n}$ is $c_{n-1}(a,b)$, therefore
\[
c_{n-1}(a,b)
=\dfrac{1}{n}
\sum_{k=0}^{\lfloor(n-1)/2\rfloor}{2n-2-j\choose n-1-j}{n-1-j\choose j}a^{n-1-2j}b^{j},
\]
which leads to (\ref{r3formulac}).
\end{proof}

\subsection{Basic equations}

First we will describe solutions $C,D$ of the equations
\begin{align}
C&=1+azC^2+bz^2C^3,\label{absequationforc}\\
D&=z+aD^2+bD^3.\label{absequationford}
\end{align}
We will assume that $a,b\in\mathbb{R}$, $b\ne0$.
Clearly, if $z\ne0$ then $C$ is a solution of (\ref{absequationforc}) if and only if
$D=zC$ is a solution of (\ref{absequationford}).

\begin{lemma}\label{lemmamultiplesolution}
Assume that $a^2+4b\ne0$ and $a^2+3b\ne0$. Then for
\[
z=z_{\pm}(a,b):=\frac{-2a^3-9ab\pm2(a^2+3b)^{3/2}}{27b^2}
\]
each equation (\ref{absequationforc}), (\ref{absequationford}) has two solutions:
\begin{align*}
C=c_{\pm}(a,b)&:=\frac{a^2+6b\pm a\sqrt{a^2+3b}}{a^2+4b},&
C&=\frac{-2a^2-3b\mp2a\sqrt{a^2+3b}}{b},\\
D=d_{\pm}(a,b)&:=\frac{-a\pm \sqrt{a^2+3b}}{3b},&
D&=\frac{-a\mp2\sqrt{a^2+3b}}{3b},
\end{align*}
the former with multiplicity two.

Moreover, equations (\ref{absequationforc}), (\ref{absequationford}) admit a multiple solution $C,D$ if and only if $z=z_{\pm}(a,b)$.
Consequently, $\mathcal{N}_{a,b}=\{z_{-}(a,b),z_{+}(a,b)\}$.
\end{lemma}

Note that
\[
z_+(a,b)\cdot z_-(a,b)=-\frac{a^2+4b}{27b^2}.
\]
If $a^2+4b=0$ then $z_-(a,b)=0$ and $z_+(a,b)=8/(27a)$.

\begin{proof}
One can check that the roots of $P_{a,b}'(w)=1-2aw-3bw^2$ are $d_{\pm}(a,b)$
and $P_{a,b}\left(d_{\pm}(a,b)\right)=z_{\pm}(a,b)$.
\end{proof}

It turns out that the solutions of (\ref{absequationford}) admit nice parametrization.

\begin{theorem}\label{theoremsolutionswitht}
For
\begin{align*}
z(a,b,t):=&\frac{-2a^3-9ab+t(3-t^2)(a^2+3b)^{3/2}}{27b^2},\\
\intertext{
$t\in\mathbb{C}$, the equation (\ref{absequationford}) has solutions}
D_0(a,b,t):=&\frac{-a+t\sqrt{a^2+3b}}{3b},\\
D_{\pm}(a,b,t):=&\frac{-2a-t\sqrt{a^2+3b}\pm\sqrt{3(4-t^2)(a^2+3b)}}{6b}.
\end{align*}
\end{theorem}

The square root $\sqrt{a^2+3b}$ must be the same
for $z(a,b,t)$, $D_0(a,b,t)$ and $D_{\pm}(a,b,t)$
and the square root $\sqrt{4-t^2}$ must be the same for $D_{\pm}(a,b,t)$.
Note that for $t_0:=a/\sqrt{a^2+3b}$ we have
\begin{equation}
z(a,b,t_0)=0,\quad
D_0\left(a,b,t_0\right)=0,\quad
D_{\pm}\left(a,b,t_0\right)=\frac{-a\pm\sqrt{a^2+4b}}{2b}.
\end{equation}

\begin{proof}
Putting $D_{\pm}:=D_{\pm}(a,b,t)$, $D_{0}:=D_{0}(a,b,t)$,
it is elementary to check that
\begin{multline*}
(D-D_+)(D-D_-)\\
=D^2+\frac{2a+t\sqrt{a^2+3b}}{3b}D+\frac{a^2+(t^2-3)(a^2+3b)+at\sqrt{a^2+3b}}{9b^2}.
\end{multline*}
Consequently,
\[
b(D-D_+)(D-D_-)(D-D_0)=b D^3+aD^2-D+z(a,b,t),
\]
which completes the proof.
\end{proof}

One can check that $D_{+}(a,b,t)=D_{-}(a,b,t)$ iff $t=\pm2$,
$D_{+}(a,b,t)=D_{0}(a,b,t)$ iff $t=1$ and
$D_{-}(a,b,t)=D_{0}(a,b,t)$ iff $t=-1$.

\subsection{The generating functions}

Now we will study the generating functions $C_{a,b}(z)$, $D_{a,b}(z)$.
Recall that they satisfy
\begin{align}
C_{a,b}(z)&=1+a z C_{a,b}(z)^2+b z^2 C_{a,b}(z)^3,\label{equationforc}\\
D_{a,b}(z)&=z+a D_{a,b}(z)^2+b D_{a,b}(z)^3,\label{equationford}
\end{align}
with $D(0)=0$. These functions
can be defined on $\mathcal{D}_{a,b}:=\mathbb{C}\setminus\left\{t z_{\pm}(a,b):t\ge1\right\}$.
If $a^2+4b=0$, $a>0$ (resp. $a<0$) then we can put $\mathcal{D}_{a,b}:=\mathbb{C}\setminus[8/(27a),+\infty)$
(resp. $\mathcal{D}_{a,b}:=\mathbb{C}\setminus(-\infty,8/(27a]$).

\begin{lemma}
For the real function $P_{a,b}(d)=d-ad^2-b d^3$ we have:
\begin{itemize}
\item If $a^2+3b\le0$ then $P_{a,b}(d)$ is an increasing function on $\mathbb{R}$.
\item If $b<0<a^2+3b$, $a>0$ then $0<d_{+}(a,b)<d_{-}(a,b)$, $P_{a,b}$ is increasing
on $(-\infty,d_{+}(a,b)]$, decreasing on $[d_{+}(a,b),d_{-}(a,b)]$ and
increasing on $[d_{-},+\infty)$.
\item If $b>0$, $a>0$ then $d_{-}(a,b)<0<d_{+}(a,b)$,
$P_{a,b}$ is decreasing on $(-\infty,d_{-}(a,b)]$, increasing on $[d_{-}(a,b),d_{+}(a,b)]$
and decreasing on $[d_{+}(a,b),+\infty)$.
\end{itemize}
\end{lemma}

Denote $\tau_{\pm}(a,b):=1/z_{\pm}(a,b)$, so that
\[
\tau_{\pm}(a,b):=\frac{2a^3+9ab\pm 2(a^2+3b)^{3/2}}{a^2+4b},
\]
with particular cases $\tau_{+}(a,-a^2/3)=3a$, $\tau_{+}(a,-a^2/4)=27a/8$.
Note that for fixed $a>0$ the function $\tau_{+}(a,b)$ increases with $b\in[-a^2/3,+\infty)$
and $\tau_{-}(a,b)$ decreases with $b\in[0,+\infty)$.

\begin{theorem}\label{theoremr3positivedefinite}
The sequence $\left\{c_n(a,b)\right\}_{n=0}^{\infty}$ is positive definite
if and only if $a^2+3b\ge0$.
If $a>0$ then for the corresponding probability distribution  $\mu(a,b)$ we have:
\begin{itemize}
\item If $b<0\le a^2+3b$ then the support of $\mu(a,b)$ is contained in $[0,\tau_{+}(a,b)]$.
\item If $b>0$ then the support of $\mu(a,b)$ is contained in $[\tau_{-}(a,b),\tau_{+}(a,b)]$.
\end{itemize}
\end{theorem}

\begin{proof}
If $a^2+3b\ge0$ then the sequence $\left\{c_n(a,b)\right\}_{n=0}^{\infty}$ is positive definite
by Theorem~\ref{theorempositivedefinite} and Lemma~\ref{lemmamultiplesolution}
and $D_{a,b}(z)$ is a Pick function on $\mathbb{C}\setminus\mathbb{R}$,
which can be extended to $(-\infty,z_{+}(a,b))$ if $b<0\le a^2+3b$
and to $\big(z_{-}(a,b),z_{+}(a,b)\big)$ if $b>0$.

If $a^2+3b<0$ then $D_{a,b}(z)$ is not analytic at $z_{\pm}(a,b)\notin\mathbb{R}$,
so can not be extended to a Pick function.
\end{proof}

\subsection{The generating functions in a neighborhood of $0$}

Applying Theorem~\ref{theoremsolutionswitht} we can write down explicit formulas.

\begin{theorem}
Assume that $a>0$, $0\ne b\in\mathbb{R}$. Then in a neighborhood of $z=0$ we have
\begin{align}
D_{a,b}(z)&=\frac{\phi(g(z))\sqrt{|a^2+3b|}-a}{3b},\label{formuladr3expicite}\\
C_{a,b}(z)&=\frac{\phi(g(z))\sqrt{|a^2+3b|}-a}{3bz},\label{formulacr3expicite}
\end{align}
where
\[
g(z):=\frac{27b^2z+2a^3+9ab}{|a^2+3b|^{3/2}}
\]
and
\[
\phi(u):=\left\{
\begin{array}{ll}
-2\sinh\left(\frac{1}{3}\mathrm{arcsinh}(u /2)\right)&\hbox{if $a^2+3b<0$,}\\
2\cosh\left(\frac{1}{3}\mathrm{arcosh}(-u/2)\right)&\hbox{if $b<0<a^2+3b$,}\\
2\sin\left(\frac{1}{3}\arcsin(u /2)\right)&\hbox{if $b>0$.}
\end{array}
\right.
\]
\end{theorem}

\begin{proof}
If $a^2+3b<0$ then
the maps
\begin{align*}
t\mapsto z(a,b,t\mathrm{i})&=\frac{-2a^3-9ab+t(3+t^2)\left|a^2+3b\right|^{3/2}}{27b^2},\\
t\mapsto D_0(a,b,t\mathrm{i})&=\frac{-a- t\sqrt{\left|a^2+3b\right|}}{3b}
\end{align*}
are real and increasing with $t\in\mathbb{R}$.
For $u=t(3+t^2)$ we have $t=2\sinh\left(\frac{1}{3}\mathrm{arcsinh}(u /2)\right)$, $t,u\in\mathbb{R}$,
which leads to (\ref{formuladr3expicite}) in this case.

If $b<0<a^2+3b$ then $t_0=a/\sqrt{a^2+3b}>1$, the functions $t\mapsto t(3-t^2)$, $t\mapsto z(a,b,t)$, $t\mapsto D_0(a,b,t)$
are decreasing for $t\ge1$ and the inverse function for the map $t\mapsto u=t(3-t^2)$, $t\ge1$, is
$t=2\cosh\left(\frac{1}{3}\mathrm{arcosh}(-u/2)\right)$, $u\le2$.

Finally, if $b>0$ then $0<t_0=a/\sqrt{a^2+3b}<1$ and for  $t\in[-1,1]$ the inverse of $u=t(3-t^2)$ is
$t=2\sin\left(\frac{1}{3}\arcsin(u /2)\right)$,  $u\in[-2,2]$, which concludes the proof.
\end{proof}

\subsection{Free cumulants}

We will use the following elementary

\begin{proposition}
For $a,b\in\mathbb{R}$ we have
\[
\frac{1}{1-az-bz^2}=\sum_{n=0}^{\infty}\kappa_n(a,b)z^n,
\]
where the coefficients $\kappa_n(a,b)$
are given by
\begin{equation}\label{kappar3}
\kappa_{n}(a,b)=\sum_{k=0}^{\lfloor n/2\rfloor}\binom{n-k}{k}a^{n-2k}b^{k}
\end{equation}
and satisfy the following recurrence: $\kappa_0(a,b)=1$,
$\kappa_1(a,b)=a$ and
\[
\kappa_{n}(a,b)=a\cdot \kappa_{n-1}(a,b)+b\cdot \kappa_{n-2}(a,b)
\]
for $n\ge2$.
\end{proposition}

\begin{corollary}
The free cumulants $\kappa_n(a,b)$, $n\ge1$, of the sequence $c_n(a,b)$ are given by (\ref{kappar3}).
If $a^2+4b>0$ then we have the following version of Binet's formula:
\begin{equation}\label{freecumulantsr3}
\kappa_n(a,b)=t_{-}u_{-}^n+t_{+}u_{+}^n
\end{equation}
where
\[
t_{\mp}=\frac{\sqrt{a^2+4b}\mp a}{2\sqrt{a^2+4b}},\qquad\qquad
u_{\mp}=\frac{\mp 2b}{\sqrt{a^2+4b}\pm a}.
\]

If $b>0$ then the corresponding probability measure $\mu(a,b)$ is infinitely divisible
with respect to the free additive convolution~$\boxplus$.
\end{corollary}

\begin{proof}
For the free $R$-transform of the sequence $c_n(a,b)$ we have
\[
R_{a,b}(z)=\frac{a+bz}{1-az-bz^2}=\frac{1}{1-az-bz^2}-1,
\]
which proves that the free cumulants are given by (\ref{kappar3}).
Moreover, it is elementary to check that if $a^2+4b>0$ then
\begin{equation}\label{freerr3decomposition}
R_{a,b}(z)=t_{-}\frac{u_{-}z}{1-u_{-}z}+t_{+}\frac{u_{+}z}{1-u_{+}z}.
\end{equation}
If $b>0$ then $t_{-},t_{+}>0$, and from (\ref{freecumulantsr3})
the sequence $\kappa_{n+2}(a,b)$, $n\ge0$, is positive definite.
This implies that $\mu(a,b)$ is infinitely divisible
with respect to the free additive convolution~$\boxplus$, see~\cite{nicaspeicherlectures}.
\end{proof}

Formula (\ref{freerr3decomposition}) implies that for $b>0$ the distribution
$\mu(a,b)$ can be decomposed as
\[
\mu(a,b)=\mu(u_{-})^{\boxplus t_{-}}\boxplus\mu(u_{+})^{\boxplus t_{+}}.
\]
where $\mu(u)^{\boxplus t}$ is dilation of the Marchenko-Pastur distribution
$\mu(1)^{\boxplus t}$ with parameter $u$.

\textbf{Examples}.
One encounters interesting sequences as free cumulants $\kappa_n(a,b)$, notably
the Fibonacci sequence $A000045$ as $\kappa_n(1,1)$,
the Pell sequence $A000129$ as $\kappa_n(2,1)$,
the	Jacobsthal sequence $A001045$ as $\kappa_n(1,2)$,
also $\kappa_n(3,-1)=A001906$ (bisection of Fibonacci sequence),
$\kappa_n(4,-1)=A001353$, $\kappa_n(5,-1)=A004254$,
$\kappa_n(5,-2)=A107839$, $\kappa_n(3,-2)=A000225$,
$\kappa_n(1,3)=A006130$, $\kappa_n(1,4)=A006131$,
$\kappa_n(1,5)=A015440$.

\subsection{Special cases.}\label{subsectionr3specialcases}
Now we will focus on some special cases: $a=0$, $a^2+3b=0$ and $a^2+4b=0$.

1. For $a=0,b=\mp1$ formula (\ref{formulacr3expicite}) yields
\[
C_{0,-1}(z)=\frac{2\sinh\left(\frac{1}{3}\mathrm{arcsinh}\left(z\sqrt{27}/2\right)\right)}{z\sqrt{3}},
\]
\[
C_{0,1}(z)=\frac{2\sin\left(\frac{1}{3}\arcsin\left(z\sqrt{27}/2\right)\right)}{z\sqrt{3}}.
\]
Comparing (\ref{fussgeneratingformula}) and (\ref{equationforc}) one can see
that $C_{0,1}(z)=\mathcal{B}_{3}(z^2)$.
In fact $c_n(0,1)$ is the aerated sequence $\binom{3n+1}{n}\frac{1}{3n+1}$,
i.e. $c_{2n}(0,1)=\binom{3n+1}{n}\frac{1}{3n+1}$ and $c_{2n+1}(0,1)=0$,
while $c_{n}(0,-1)$ is the aerated sequence $\binom{3n+1}{n}\frac{(-1)^n}{3n+1}$.
By Theorem~\ref{theoremr3positivedefinite} the sequence $c_n(0,1)$ is positive definite, $c_n(0,-1)$ is not.
It is known that
\[
\binom{3n+1}{n}\frac{1}{3n+1}=\int_{0}^{27/4} x^n W_{3,1}(x)\,dx,
\]
$n=0,1,2,\ldots$, where
\[
W_{3,1}(x)=\frac{3\left(1+\sqrt{1-4x/27}\right)^{2/3}-2^{2/3}x^{1/3}}
{2^{4/3}3^{1/2}\pi x^{2/3}\left(1+\sqrt{1-4x/27}\right)^{1/3}},
\]
see \cite{pensonsolomon2002} and (41) in \cite{mlopezy2013}. This implies, that
\[
c_n(0,1)=\int_{-\sqrt{27}/2}^{\sqrt{27}/2} x^n W_{3,1}(x^2)|x|\,dx,
\]
$n=0,1,2,\ldots$, so that the distribution $\mu(0,1)$ is absolutely continuous,
with the density function $W_{3,1}(x^2)|x|$ on $(-\sqrt{27}/2,\sqrt{27}/2)$.
Moreover, in view of Corollary~5.3 in \cite{mlotkowski2010}, $\mu(0,1)$
is infinitely divisible with respect to the free convolution ``$\boxplus$",
which means that the sequence $c_n^{\boxplus t}(0,1)$ is positive definite for every $t>0$.
Moreover, since for $t>0$ the distribution $\mu(0,1)^{\boxplus t}$ is symmetric,
also the sequence $c_{2n}^{\boxplus t}(0,1)$ is positive definite and the support
of the corresponding probability distribution is contained in $[0,+\infty)$.

2. For $b=-a^2/3$ it is easy to check that
\[
D_{a,-a^2/3}(z)=\frac{1-(1-3az)^{1/3}}{a}
\]
and
\begin{equation}\label{formulac3minus3}
c_n(a,-a^2/3)=3\binom{1/3}{n+1}(-3a)^n=\frac{a^n}{(n+1)!}\prod_{i=1}^{n}(3i-1).
\end{equation}
In particular $c_n(3,-3)=A097188$ (see also $A025748$):
\[
1, 3, 15, 90, 594, 4158, 30294, 227205, 1741905, 13586859, 107459703,\ldots,
\]
while $\kappa_n(3,-3)=A057083$.
In view of Theorem~\ref{theoremr3positivedefinite} the sequence $c_n(a,-a^2/3)$
is positive definite. More precisely, if $a>0$ then
\begin{equation}
c_n(a,-a^2/3)=\int_{0}^{3a}x^n\frac{\sqrt{3}\left(3ax^2-x^3\right)^{1/3}}{2a\pi x}\,dx,
\end{equation}
$n=0,1,2,\ldots$.
Indeed, from the definition of the beta function the integral is equal to
\[
\frac{a^n 3^{n+3/2}}{2\pi}\frac{\Gamma(n+2/3)\Gamma(4/3)}{\Gamma(n+2)}.
\]
Using identities: $\Gamma(x+1)=x\Gamma(x)$ and $\Gamma(1/3)\Gamma(2/3)=2\pi/\sqrt{3}$
one can check that this coincides with the right hand side of (\ref{formulac3minus3}).

3. Put $b=-a^2/4$. Since
\[
1+azB^4-a^2z^2B^6/4-B^2=(1+B-azB^3/2)(1-B+azB^3/2)
\]
we see that $C_{a,-a^2/4}(z)$ can be represented as $B(z)^2$,
where $B(z)$ satisfies equation
\[
B(z)=1+ az B(z)^3/2.
\]
This means that
\[
C_{a,-a^2/4}(z)=\mathcal{B}_{3}(az/2)^2
\]
and
\[
c_n(a,-a^2/4)=\left(\frac{a}{2}\right)^n\binom{3n+2}{n}\frac{2}{3n+2},
\]
in particular $c_n(2,-1)=A006013$.
The density of $\mu(2,-1)$,
can be found in \cite{pezy}, formula (29) and \cite{mlopezy2013}, formula(42).
The cumulant sequence $\kappa_n(2,-1)$ is particularly simple: $2,3,4,5,\ldots$ ($A000027$ in OEIS).

\section{The case $r=4$}\label{sectionr4}

Put $\mathbf{a}:=(a,b,e)\in\mathbb{R}^3$, with $e\ne0$, $P_{a,b,e}(w)=w-aw^2-bw^3-ew^4$.
We have $c_0(\mathbf{a})=1$, $c_1(\mathbf{a})=a$, $c_2(\mathbf{a})=2 a^2 + b$, $c_3(\mathbf{a})=5 a^3 + 5 a b + e$,
$c_4(\mathbf{a})=14 a^4 + 21 a^2 b + 3 b^2 + 6 a e$, which leads to
% necessary conditions for positive definiteness of the sequence $c_n(\mathbf{a})$: $a^2 + b\ge0$ and

\begin{proposition}\label{propositionr4necessary}
If the sequence $c_n(a,b,e)$ is positive definite then
\begin{equation}\label{necessaryr4}
a^6 + 3 a^4 b + 3 a^2 b^2 + 2 b^3 - 2 a b e - e^2\ge0.
\end{equation}
\end{proposition}

Since
\[
a^6 + 3 a^4 b + 3 a^2 b^2 + 2 b^3 - 2 a b e - e^2
=(a^2 + b)^3+(a^2 + b)b^2 - (a b + e)^2,
\]
inequality (\ref{necessaryr4}) implies that $a^2 + b\ge0$.

\subsection{Real roots.}
Now we will provide a sufficient condition.

\begin{theorem}\label{theoremr4realroots}
The polynomial $P'_{a,b,e}(w)$ has only real roots if and only if
\begin{equation}\label{sufficientr4}
9 a^2 b^2 + 27 b^3 - 32 a^3 e - 108 a b e - 108 e^2\ge0.
\end{equation}
If this is the case then $c_n(a,b,e)$ is positive definite.
\end{theorem}

%We will see from (\ref{sufficientr4equiv}) that (\ref{sufficientr4}) implies that $3b^2\ge8ae$.

\begin{proof}
It is elementary to see that
\[
P'_{a,b,e}(w)=1-2aw-3bw^2-4ew^3
\]
has only real roots if and only if
\[
P''_{a,b,e}(w)=-2(a+3bw+6ew^2)
\]
has two real roots $w_{-},w_{+}$
such that either $\mathrm{sgn}(P'_{a,b,e}(w_{-}))\ne\mathrm{sgn}(P'_{a,b,e}(w_{+}))$
or $w_{-}=w_{+}$ and $P'_{a,b,e}(w_{-})=0$.
Then $3 b^2 \ge 8 a e$ and
\[
w_{\pm}=\frac{-3b\pm\sqrt{9b^2-24ae}}{12e}.
\]
Since
\[
P'_{a,b,e}(w_{\pm})=\frac{36abe-9b^3+72e^2\pm\sqrt{3}(3b^2-8ae)^{3/2}}{72e^2},
\]
we have $P'_{a,b,e}(w_{-})\le0\le P'_{a,b,e}(w_{+})$ if and only if
\begin{equation}\label{sufficientr4equiv}
(36abe-9b^3+72e^2)^2\le 3(3b^2-8ae)^3,
\end{equation}
which is equivalent to (\ref{sufficientr4}) and implies that $3b^2\ge8ae$.
\end{proof}

\textbf{Remark.} If the sequence $c_n(a,b,e)$ is positive definite then, by Proposition~\ref{propositionrtransform},
the free cumulants of the corresponding probability distribution $\mu(a,b,e)$
are the coefficient in the Taylor expansion
\[
\frac{a z+b z^2+e z^3}{1-a z-b z^2-c z^3}=\sum_{n=1}^{\infty}\kappa_n(a,b,e)z^n.
\]
One can check that $\det \left(\kappa_{i+j+2}(a,b,e)\right)_{i,j=0}^2=-e^4$,
so if $e\ne0$ then $\mu(a,b,e)$ is never infinitely divisible with respect
to the additive free convolution~$\boxplus$.

\subsection{A special subclass}

Now we will confine ourselves to the case when the coefficients satisfy
\begin{equation}\label{r4complexroots}
b^3=4abe+8e^2.
\end{equation}
In particular we will describe those parameters $a,b,e$ for which
$P'_{a,b,e}(w)$ admits complex roots but $\mathcal{N}_{a,b,e}\subseteq\mathbb{R}$.

\begin{proposition}\label{propositionr4}
Assume that $b^3=4abe+8e^2$.
Then
\begin{align}
P_{a,b,e}(w)&=-\frac{w(2ew+b)\left(2bew^2+b^2w-4e\right)}{4be},\label{r4factorp}\\
P'_{a,b,e}(w)&=-\frac{(4ew+b)\left(2bew^2+b^2w-2e\right)}{2be}\label{r4factorpp}
\end{align}
and
\[
\mathcal{N}_{a,b,e}=\left\{
\frac{-b^4-32be^2}{256e^3},\frac{e}{b^2}\right\}. %\subseteq\mathbb{R}.
\]
In particular, $c_n(a,b,e)$ is positive definite.

On the other hand, if the polynomial $P'_{a,b,e}(w)$ admits complex roots
$w_1,\overline{w_1}\in\mathbb{C}\setminus\mathbb{R}$ which
satisfy $P_{a,b,e}(w_1)=P_{a,b,e}(\overline{w_1})\in\mathbb{R}$ then  $b^3=4abe+8e^2$ and $b^4+16be^2<0$.
\end{proposition}

Note that if $e>0$, $b\ne0$ then $(-b^4-32be^2)/(256e^3)\le e/b^2$, %$\frac{-b^4-32be^2}{256e^3}\le\frac{e}{b^2}$
 with equality only if $b^3=-16e^2$.

\begin{proof}
It is easy to check that if $b^3=4abe+8e^2$ then (\ref{r4factorp}), (\ref{r4factorpp}) hold
and for the roots
\[
\frac{-b}{4e},\qquad w_{\pm}:=\frac{-b^2\pm\sqrt{b^4+16be^2}}{4be}
\]
of $P'_{a,b,e}(w)$ we have
\[
P_{a,b,e}\left(\frac{-b}{4e}\right)=\frac{-b^4-32be^2}{256e^3},\qquad
P_{a,b,e}\left(w_{\pm}\right)=\frac{e}{b^2}.
\]

Now assume that $w_0,w_1,w_2$ are the roots of $P'_{a,b,e}(w)$, so that
\[
1-2aw-3bw^2-4ew^3=-4e(w-w_0)(w-w_1)(w-w_2).
\]
This implies that $w_0,w_1,w_2\ne0$ and
\begin{align*}
a&=\frac{w_0 w_1+w_0 w_2+w_1 w_2}{2w_0 w_1 w_2},\\
b&=-\frac{w_0+w_1+w_2}{3w_0 w_1 w_2},\\
e&=\frac{1}{4w_0 w_1 w_2}.
\end{align*}
Then
\begin{align*}
P_{a,b,e}(w_1)&=\frac{w_1}{2}-\frac{w_1^2}{6w_0}-\frac{w_1^2}{6 w_2}+\frac{w_1^3}{12w_0 w_2}\\
&=\frac{w_1(w_1^2-2w_0 w_1+6w_0 w_2-2w_1 w_2)}{12w_0 w_2}.
\end{align*}

Now assume that $w_0\in\mathbb{R}$, $w_1=x+y\mathrm{i}$, $w_2=\overline{w_1}=x-y\mathrm{i}$, $x,y\in\mathbb{R}$, $y\ne0$. Then
%$x\ne0$ and
\[
\Im(w_1^2(w_1^2-2w_0 w_1+6w_0 w_2-2w_1 w_2))=8y^3(w_0-x)
\]
so that if $P_{a,b,e}(w_1)\in\mathbb{R}$ then $w_0=x$ and
\[
a=\frac{3x^2+y^2}{2x(x^2+y^2)},\quad
b=\frac{-1}{x^2+y^2},\quad
e=\frac{1}{4x(x^2+y^2)}.
\]
This implies $b<0$, $b^3+16e^2>0$ and $b^3=4abe+8e^2$.
\end{proof}

\begin{figure}%[h]
    \centering
    \includegraphics[width=0.8\textwidth]{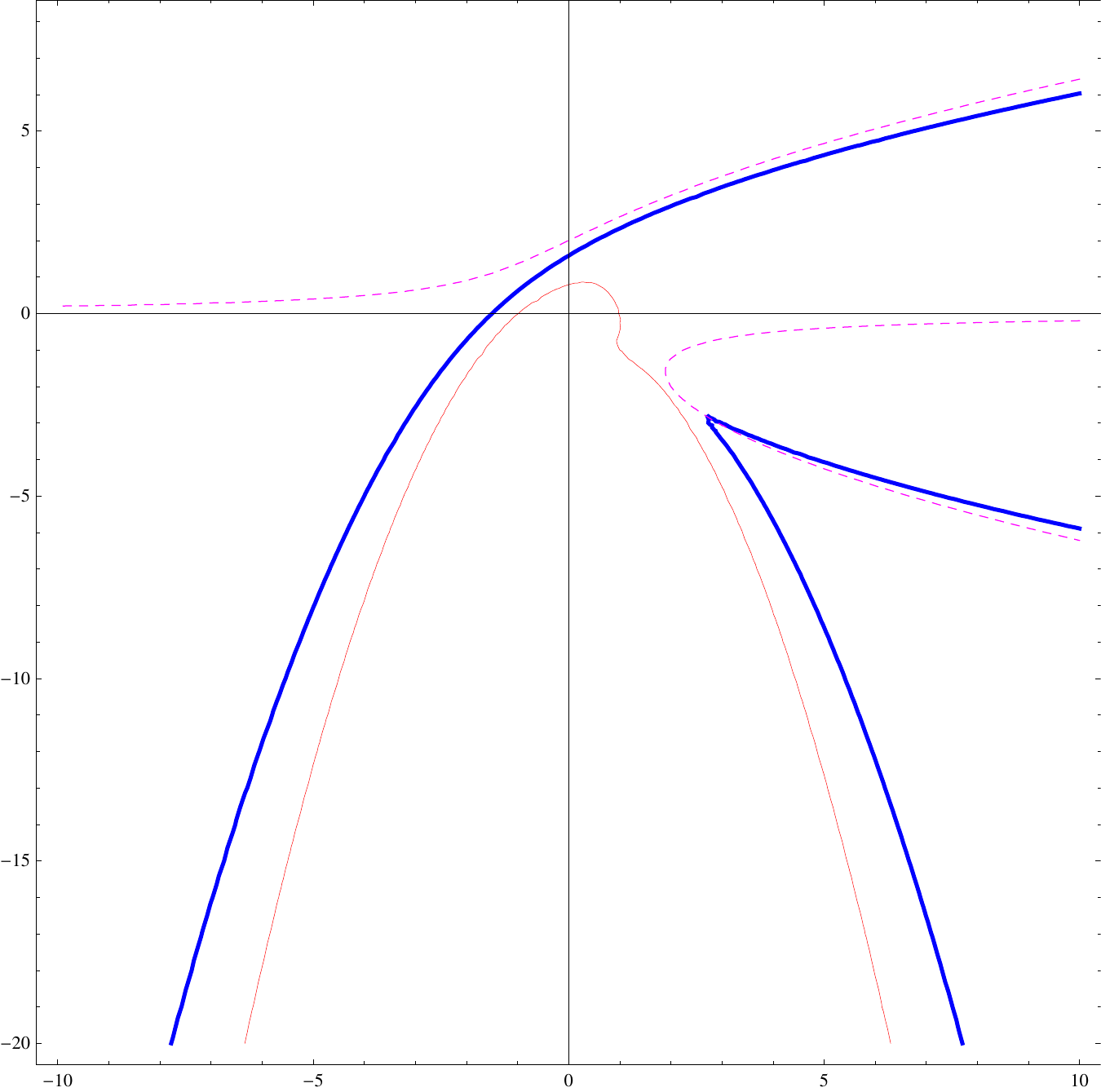}
    \caption{Here we illustrate conditions (\ref{necessaryr4}), (\ref{sufficientr4}) and (\ref{r4complexroots})
    for $e=1$ and $-10\le a\le10$, $-20\le b\le8$.
     The thin red line represents the equation: $a^6 + 3 a^4 b + 3 a^2 b^2 + 2 b^3 - 2 a b  - 1=0$, the thick blue: $9 a^2 b^2 + 27 b^3 - 32 a^3 - 108 a b - 108=0$,  and the dashed magenta: $b^3=4ab+8$.}
    \label{figr4}
\end{figure}

On Figure~\ref{figr4} we illustrate the conditions (\ref{necessaryr4}), (\ref{sufficientr4}) and (\ref{r4complexroots})
for $e=1$. If $(a,b)$ is below the thin red line then $c_n(a,b,1)$ is not positive definite.
If $(a,b)$ is either above the left piece or south-east of the right piece of the thick blue line then
the roots of $P'_{a,b,1}(w)$ are real and $c_n(a,b,1)$ is positive definite.
On the dashed magenta line there are points $(a,b)$ which satisfy $b^3=4ab+8$.
For some of these points, namely when $-16<b^3<0$, the polynomial $P'_{a,b,1}(w)$ has complex nonreal roots,
but still $\mathcal{N}_{a,b,1}\subseteq\mathbb{R}$ and therefore $c_n(a,b,1)$ is positive definite.

%\textbf{Example} (A121988).
%Take $(a,b,e)=(2,-2,1)$. Then $P(w)=w-2w^2+2w^3-w^4$,
%$P'(w)=1-4w+6w^2-4w^3$. The roots of $P'(w)$ are $1/2,1/2\pm i/2$
%and $P(1/2\pm i/2)=1/4$.

It turns out that if $b^3=4abe+8e^2$ then we are able to write down the generating function.

\begin{proposition}
If $b^3=4abe+8e^2$ then
\begin{equation}\label{r4specialdecompositionformula}
\begin{gathered}
w-aw^2-bw^3-ew^4-z\\
=-e \left(w^2+\frac{bw}{2e}-\frac{1}{b}-\frac{\sqrt{e^2 - b^2 ez}}{be}\right)
\left(w^2+\frac{bw}{2e}-\frac{1}{b}+\frac{\sqrt{e^2 - b^2 ez}}{be}\right).
\end{gathered}
\end{equation}
Consequently,
\begin{equation}\label{r4dformula}
D_{a,b,e}(z)=\frac{-b^2+\sqrt{b^4+16be^2-\mathrm{sgn}(e) 16be\sqrt{e^2-b^2 ez}}}{4be},
\end{equation}
\begin{equation}%\label{r4dformula}
C_{a,b,e}(z)=\frac{-b^2+\sqrt{b^4+16be^2-\mathrm{sgn}(e) 16be\sqrt{e^2-b^2 ez}}}{4bez}
\end{equation}
for $z$ in a neighborhood of $0$.
\end{proposition}

\begin{proof}
It is easy to verify (\ref{r4specialdecompositionformula}). This implies that for $z$
in a neighborhood of $0$
\[
D_{a,b,e}(z)^2+\frac{b D_{a,b,e}(z)}{2e}-\frac{1}{b}+\mathrm{sgn}(e)\frac{\sqrt{e^2-b^2 ez}}{be}=0,
\]
as $D_{a,b,e}(0)=0$, which, in turn, leads to (\ref{r4dformula}).
\end{proof}

Finally we observe that the case $b^3=4abe+8e^2$ comes from the monotonic convolution.
Indeed, \[(\alpha_1)\rhd(\alpha_2)=(\alpha_1+\alpha_2,-2\alpha_1 \alpha_2,\alpha_1^2 \alpha_2)\]
and every triple $(a,b,e)\in\mathbb{R}^3$, satisfying $b^3=4abe+8e^2$, $b,e\ne0$,
can be represented as $(\alpha_1+\alpha_2,-2\alpha_1 \alpha_2,\alpha_1^2 \alpha_2)$ for some $\alpha_1,\alpha_2\in\mathbb{R}\setminus\{0\}$.

%\textbf{Example} (A048779).
%Take $(a,b,e)=(3,-4,2)$ (intersection of curves) $P(w)=w-3w^2+4w^3-2w^4$.
%Then $P'(w)=1-6w+12w^2-8w^3=(1-2w)^3$.
%In this case \[C(z)=\frac{1-(1-8z)^{1/4}}{2z}.\]
%% Karol Penson-the measure

\section{Symmetric case}\label{sectionsymmetric}

In this section we assume that $a_j=0$ whenever $j$ is even, i.e. that the polynomial $P_{\mathbf{a}}(w)$ is odd.
Then $D_{\mathbf{a}}(z)$ is odd as well, $C_{\mathbf{a}}(z)$ is even and
$c_{n}(\mathbf{a})=0$ whenever $n$ is odd. If in addition the sequence $c_n(\mathbf{a})$ is positive definite
then the corresponding probability distribution $\mu(\mathbf{a})$ is symmetric,
i.e. $\mu(\mathbf{a})(X)=\mu(\mathbf{a})(-X)$ for every Borel set $X\subseteq\mathbb{R}$.
The case $r=3$ was already studied in Subsection~\ref{subsectionr3specialcases}.

\subsection{The case $r=5$.} Assume that $\mathbf{a}=(0,a,0,b)$, $b\ne0$, $P_{\mathbf{a}}(w)=w-aw^3-bw^5$.
Then a necessary condition for positive definiteness of $c_{n}(0,a,0,b)$ is $a>0$ and $2 a^2 + b\ge0$.

\begin{proposition}
If $a>0$, $b<0$ and $9a^2+20b\ge0$ then the sequence $c_n(0,a,0,b)$ is positive definite.
\end{proposition}

\begin{proof}
We have $P_{0,a,0,b}'(w)=1-3a w^2-5b w^4$, so $w_0$ is a root of $P_{0,a,0,b}'(w)$ if and only if
\[
w_0^2=\frac{3a\pm\sqrt{9a^2+20b}}{-10b}.
\]
If $a>0$, $b<0$ and $9a^2+20b\ge0$ then all these roots are real.
\end{proof}

Example: Take $\mathbf{a}=(0,2,0,-1)$.
Then $C_{\mathbf{a}}(z)$ satisfies $C_{\mathbf{a}}(z)=1+2z^2 C_{\mathbf{a}}(z)^{3}-z^4 C_{\mathbf{a}}(z)^{5}$,
equivalently: $C_{\mathbf{a}}(z)\big(1-z^2 C_{\mathbf{a}}(z)^2\big)^2=1$.
Putting $C_{\mathbf{a}}(z):=B(z)^2$, we get $B(z)=1+z^2 B(z)^5$, which means that $B(z)=\mathcal{B}_{5}(z^2)$.
Therefore $C_{\mathbf{a}}(z)=\mathcal{B}_{5}(z^2)^2$ and, consequently,
$c_{2n}(0,2,0,-1)=\binom{5n+2}{n}2/(5n+2)$ ($A118969$).

\subsection{The case $r=7$.} Let us now consider $\mathbf{a}=(0,a,0,b,0,e)$, with $e\ne0$.
It is easy to check that if $c_n(0,a,0,b,0,e)$ is positive definite then $a\ge0$.

\begin{proposition}
If $a\ge0$, $b\le0$, $e>0$ and
\begin{equation}\label{symmetric7inequality}
225 a^2 b^2 + 500 b^3 - 756 a^3 e - 1890 a b e - 1323 e^2\ge0
\end{equation}
then the sequence $c_n(0,a,0,b,0,e)$ is positive definite.
\end{proposition}

\begin{proof}
It is sufficient to prove that (\ref{symmetric7inequality}) implies that the polynomial
\[
Q(t):=1-3at-5bt^2-7et^3
\]
has three positive roots.
%Equivalently, as $e<0$, $Q'(t)$ has roots $0\le t_-\le t_+$ such that

First we note that (\ref{symmetric7inequality}) is equivalent to
\begin{equation}\label{symmetric7inequalityaux}
4(25 b^2 - 63 a e)^3 \ge (945abe-250 b^3+1323e^2)^2,
\end{equation}
hence implies that $25 b^2 - 63 a e\ge0$.
Therefore the derivative $Q'(t)=-3 a - 10 b t - 21 e t^2$ has two real roots:
\[
t_{\pm}=\frac{-5 b \pm\sqrt{25 b^2 - 63 a e}}{21 e},
\]
with $0\le t_{-}\le t_{+}$, and we have
\[
Q(t_{\pm})=\frac{945abe-250b^3+1323e^2\pm2(25b^2-63ae)^{3/2}}{1323 e^2}
\]
If (\ref{symmetric7inequalityaux}) holds with sharp inequality then we have $1=Q(0)>0$, $Q(t_{-})<0$, $Q(t_{+})>0$ and $\lim_{t\to\infty}Q(t)=-\infty$,
so there are $t_1,t_2,t_3$ such that $0<t_1<t_{-}<t_2<t_{+}<t_3$ such that $Q(t_1)=Q(t_2)=Q(t_3)=0$.
Now we recall that positive definiteness is preserved by pointwise limits.
\end{proof}

\section{Examples}\label{sectionexamples}

In this section we provide a record of examples of sequences $c_n(\mathbf{a})$
for which we can verify positive definiteness.
Most of them appear in OEIS, possibly with an additional $0$ or $1$ term at the beginning.
We refer also to Table~5 in~\cite{mathar2012}.

\subsection{Patalan numbers.}

Patalan numbers of order $p\in\mathbb{R}\setminus\{0\}$ (see \cite{richardson2015}) are defined
by the generating function
\begin{equation}\label{patalangenerating}
A_p(z)=\sum_{n=0}^{\infty}\mathrm{pat}_n(p)z^n:=\frac{1-\left(1-p^2 z\right)^{1/p}}{pz},
\end{equation}
so that
\[
\mathrm{pat}_n(p):=-p^{2n+1}\binom{n-1/p}{n+1}.
\]
Now we observe

\begin{proposition}
The sequence $\mathrm{pat}_n(p)$ is positive definite if and only if $|p|\ge1$.
Moreover, if $|p|>1$, $n\ge0$ then
\begin{equation}\label{patalanintegral}
%-p^{2n+1}\binom{n-1/p}{n+1}=
\mathrm{pat}_n(p)=\int_{0}^{p^2} x^n\frac{(p^2-x)^{1/p}\sin(\pi/p)}{p\pi x^{1/p}}\,dx.
\end{equation}
\end{proposition}

\begin{proof}
One can verify (\ref{patalanintegral}) by using the definitions
and basic properties of the beta and gamma functions.
On the other hand
\[
\mathrm{pat}_0(p)\mathrm{pat}_2(p)-\mathrm{pat}_1(p)^2=p^2(p^2-1)/12,
\]
so the sequence is not positive definite definite if $|p|<1$.
\end{proof}

Formula (\ref{patalangenerating}) implies that $z A_p(z)$ is the inverse function of
\[
P(w):=\frac{1-\left(1-p w\right)^p}{p^2}.
\]
In particular, if $p\ge2$ is an integer then $A_p(z)$ coincides with $C_{\mathbf{a}}(z)$,
where
\[
\mathbf{a}=(a_2,\ldots,a_p),\qquad a_k=\frac{(-p)^k}{p^2}\binom{p}{k}.
\]
Since $P'(w)=(1-pw)^{p-1}$ the positive definiteness of the corresponding sequence $c_n(\mathbf{a})$ is
also a consequence of Corollary~\ref{corollarypositivedefinite}.
For $p=2,3,4,5,6,7,8,9,10$
we obtain $A000108$ (Catalan numbers),
$A097188$ (or $A025748$), $A025749$, $A025750$, $A025751$, $A025752$, $A025753$.
$A025754$, $A025755$,
respectively.
The dilated version of $p=4$, with $\mathbf{a}=(3,-4,2)$, leads to $A048779$.
Let us also note that for $p=-2,-3$, $-4,-5,-6$, $-7,-8,-9,-10$ the sequence $\mathrm{pat}_n(p)$
appears in OEIS as $A001700$, $A034171$, $A034255$, $A034687$, $A034789$, $A034904$, $A034996$,
$A035097$, $A035323$, respectively.

\subsection{Fuss numbers.}
Assume that $p\ge2$ is an integer.
From (\ref{fussgeneratingformula}) we have $\mathcal{B}_{p}(z^{p-1})=C_{\mathbf{a}}(z)$
for $\mathbf{a}=(0,0,\ldots,0,1)$, with $p-2$~zeros. Therefore the Fuss numbers $\binom{np+1}{n}\frac{1}{np+1}$
of order $p$ appear as the nonzero terms in the sequence $c_n(\mathbf{a})$.

More interestingly, put $B(z):=\mathcal{B}_{p}(z)^{p-1}$, the generating function
of the sequence $\binom{np+p-1}{n}\frac{p-1}{np+p-1}$.
As a consequence of (\ref{fussgeneratingformula}) it
satisfies equation $B(z)\big(1-zB(z)\big)^{p-1}=1$, so that
$zB(z)$ is the composition inverse of $P(w)=w(1-w)^{p-1}$.
Therefore
\[
\binom{np+p-1}{n}\frac{p-1}{np+p-1}=c_n(a_2,\ldots,a_p),\qquad
a_k=\binom{p-1}{k-1}(-1)^k.
\]
Since $P'(w)=(1-pw)(1-w)^{p-2}$, Corollary~\ref{corollarypositivedefinite}
implies that this sequence is positive definite, which is already known \cite{mlotkowski2010,mlopezy2013,mlotkowskipenson2014binomial,forresterliu2014,liupego2014}.
Putting $p=2,3,4,5,6,7,8,9,10,11$ we get $A000108$ (Catalan numbers),
$A006013$, $A006632$, $A118971$, $A130564$, $A130565$, $A234466$, $A234513$, $A234573$, $A235340$
respectively.

\subsection{Various examples for $r=3$.}

Recall that by Theorem~\ref{theoremr3positivedefinite} the sequence $c_n(a,b)$ is positive definite if and only if $a^2+3b\ge0$.

Positive definite sequences:
$c_n(1,1)=A001002$,
$c_n(1,2)=A250886$,
$c_n(1,3)=A276314$,
$c_n(2,1)=A192945$,
$c_n(2,2)=A276310$,
$c_n(2,3)= A250887$,
$c_n(2,-1)= A006013$, %$\binom{3n+1}{n}/(n+1)$
$c_n(3,2)=A276315$,
$c_n(3,3)=A295541$ (up to a sign),
$c_n(3,4)=A250888$,
$c_n(3,-1)=A249924$,
$c_n(3,-2)=A085614$, %%elementary arches of size n
$c_n(3,-3)=A097188$ (and also $A025748$), %%3rd order Patalan numbers
$c_n(4,-1)=A276316$,
$c_n(4,-3)=A250885$.

Also the following free powers are positive definite (see Proposition~\ref{propositionfreepower}):
$c_n^{\boxplus 2}(2,1)=A228966$,
$c_n^{\boxplus 3}(2,2)=A231554$.
In view of remarks in Subsection~\ref{subsectionr3specialcases}.1
the following sequences are positive definite: $c_{2n}^{\boxplus 2}(0,1)=A027307$,
$c_{2n}^{\boxplus 3}(0,1)=A219535$,
$2^n\cdot c_{2n}^{\boxplus 1/2}(0,1)=A003168$,
$2^n\cdot c_{2n}^{\boxplus 3/2}(0,1)=A219536$.
These equalities can be verified by comparing (\ref{freepowerformulad})
with equations for the corresponding generating functions provided in OEIS.

Not positive definite sequences:
$c_n(-1,-1)=A103779$,
$c_n(-1,-2)=A217361$.

\subsection{Examples for $r=4$.}

By Theorem~\ref{theoremr4realroots} the following sequences are positive definite:
$c_n(1,1,-1)=A063020$, %interesting
$c_n(2,0,-1)=A236339$,
$c_n(2,3,1)=A214692$,
$c_n(3,-3,1)=A006632$,
$c_n(5,-8,4)=A024492$. %$\varsubsetneq$
The sequence
$c_n(2,-2,1)=A121988$ (see also $A129442$) %interesting: Number of vertices of the n-th multiplihedron.
is positive definite by Proposition~\ref{propositionr4}.

Using Proposition~\ref{propositionr4necessary} or checking Hankel determinants
one can check that the following sequences are not positive definite:
$c_n(1,0,1)=A049140$,
$c_n(1,0,-1)=A063033$,
$c_n(-2,0,-1)=A217362$,
$c_n(0,1,1)=A217358$,
$c_n(0,-1,-1)=A217359$,
$c_n(1,1,1)=A063018$,
$c_n(1,-1,1)=A063019$,
$c_n(3,3,1)=A192946$,
$c_n(-2,0,-1)=A217362$,
%$c_n(2,1,1)$ not,
%$c_n(2,2,1)$,
$c_n(2,1,2)=A214372$,
$c_n(0,2,1)=A055392$, %interesting
$c_n(2,0,-1)=A236339$. %interesting

\subsection{Monotonic powers.}
Define $\mathbf{a}(k):=(1)^{\rhd k}$, so that
$P_{\mathbf{a}(1)}(w)=P_{1}(w)=w-w^2$ and
\[P_{\mathbf{a}(k)}=P_{1}\circ P_{1}\circ\ldots\circ P_{1}(w),%\qquad\hbox{$k$ composition factors,}
\]
$k$ composition factors, for example
\begin{align*}
P_{\mathbf{a}(2)}(w)&=w - 2 w^2 + 2 w^3 - w^4,\\
P_{\mathbf{a}(3)}(w)&=w - 3 w^2 + 6 w^3 - 9 w^4 + 10 w^5 - 8 w^6 + 4 w^7 - w^8.
\end{align*}
In view of Corollary~\ref{corollarymonotonic} all the sequences $c_n(\mathbf{a}(k))$
are positive definite.
In particular: %the following sequences are positive definite:
$c_n(\mathbf{a}(2))=A121988$ (and $A129442$),
$c_n(\mathbf{a}(3))=A158826$, $c_n(\mathbf{a}(4))=A158827$, $c_n(\mathbf{a}(5))=A158828$.

\end{document}